\pdfoutput=1
\documentclass[11pt, a4paper]{article}

\usepackage{amssymb}
\usepackage{amsmath}
\usepackage{amsthm}
\usepackage{mathtools}
\usepackage{cases,color,fancybox}
\usepackage{wrapfig,ascmac}
\usepackage{graphicx,xcolor}
\usepackage{tikz}
\usepackage{pgfplots}
\usepackage{here}
\usepackage{cleveref}

\crefname{equation}{}{}

\newtheorem{Th}{Theorem}[section] 
\newtheorem{Lem}[Th]{Lemma} 
\newtheorem{Prop}[Th]{Proposition}
 
\newtheorem{definition}[Th]{Definition} 
\newtheorem{problem}[Th]{Problem}

\newcommand{\R}{{\mathbb R}}

\newcommand{\N}{{\mathbb N}}
\newcommand{\Sym}{{\mathcal S}}

\newcommand{\ov}[1]{\overline{#1}}
\newcommand{\Tp}{{\mathrm T}}
\newcommand{\tr}{\operatorname{tr}}
\newcommand{\defeq}{\coloneqq}
\newcommand{\prz}[2]{ \frac{\partial{#1}}{\partial{#2}} }

\newcommand{\dvg}{\operatorname{div}}

\newcommand{\into}{\int_\Omega}

\newcommand{\DS}{\displaystyle}

\newcommand{\dt}{{\Delta t}}
\newcommand{\E}{{\mathcal{E}}}
\newcommand{\Lo}{{L^2(\Omega)}}

\newcommand{\bm}[1]{\textbf{\textit{#1}}}
\newcommand{\Hv}{{\bm{H}}}
\newcommand{\Vv}{{\bm{V}}}
\newcommand{\Hs}{{H}}
\newcommand{\blaket}[2]{ \left\langle {#1},{#2}\right\rangle_{\Vv^*,\Vv} }
\newcommand{\inner}[3]{ \left( {#1},{#2}\right)_{#3} }

\newcommand{\Pop}[1]{\mathcal{P}_{#1}}
\newcommand{\sigmas}[1]{{\sigma_{#1}^*}}

\newcommand{\eps}{\varepsilon}

\providecommand{\keywords}[1]
{
  \small	
  \textbf{Keywords: } #1
}

\providecommand{\MSC}[1]
{
  \small	
  \textbf{2020 Mathematics Subject Classification: } #1
}

\title{
A Projection Method for an Elasto-plasticity Model \\with Linear Kinematic Hardening}

\author{
    Yoshiho Akagawa\thanks{Department of Mathematics, Kyoto University of Education, 1 Fujinomori, Fukakusa, Fushimi-ku, Kyoto 612-8522, Japan (akagawa@kyokyo-u.ac.jp)}
    ~and 
    Kazunori Matsui\thanks{Corresponding author. Department of Logistics and Information Engineering, Tokyo University of Marine Science and Technology, 2-1-6 Etchujima, Koto-ku, Tokyo 135-0044, Japan (kmat002@kaiyodai.ac.jp)}
}
\date{}

\begin{document}
    
\maketitle

\begin{abstract}

    We consider a dynamical elasto-plasticity system with Kelvin--Voigt viscosity and linear kinematic hardening of Melan--Prager type.
    The model is formulated in a variational framework in which a constraint set for the stress evolves in time and is translated by an internal (backstress) variable.
    As a consequence, the flow rule is coupled with an equation of motion through a quasi-variational structure, since the constraint set depends on the unknown internal variable.

    To construct solutions, we employ Rothe's method and introduce a projection-based time discretization.
    Each time step consists of solving a linear viscous-elastic subproblem to obtain a trial stress, followed by a projection onto the translated constraint set.
    We establish stability of the resulting discrete solutions under suitable norms.
    By compactness and passage to the limit as the time step tends to zero, we prove existence of a weak solution in the variational sense, and uniqueness follows from an energy argument.

    The results cover time-dependent yield bounds without assuming spatial continuity or a strictly positive lower bound, and the discretization provides a constructive basis for numerical approximation.
\end{abstract}


\MSC{34A60, 35K61, 35D30, 65M12, 74C05}

\section{Introduction}
\label{sec:intro}

\subsection{Plasticity and linear kinematic hardening}

The mechanical behaviour of many engineering materials such as metals is characterized by both elastic and plastic responses. Elasticity denotes the reversible deformation that disappears after unloading, while plasticity denotes the irreversible deformation that remains. Let $T>0$ be fixed. We consider a body occupying a bounded Lipschitz domain $\Omega\subset\R^d$ ($2\le d\in\N$) we denote the displacement by $u:[0,T]\times\Omega\to\R^d$, the total strain by $\E(u)=\frac{\nabla u + (\nabla u)^\Tp}{2}$, and the stress by $\sigma:[0,T]\times\Omega\to\Sym_d$, where $\Sym_d$ is the space of real symmetric $d\times d$ matrices.

It is common to decompose the total strain into an elastic part and a plastic part,
\[
    \E(u)=\eps_e+\eps_p,
\]
where the elastic strain $\eps_e$ is recoverable under unloading and the plastic strain $\eps_p$ is the residual deformation. For linear isotropic elasticity we write
\[
    \eps_e=S\sigma,
\]
with a homogeneous, symmetric, and isotropic compliance operator $S:\Sym_d \rightarrow \Sym_d$, i.e.,
\begin{align}\label{def:S}
    S\tau=\frac{1+\nu}{E}\tau-\frac{\nu}{E}(\tr\tau)E_d,
    \qquad E>0,\; -1<\nu<\frac{1}{d-1},
\end{align}
where $E_d$ denotes the identity matrix of order $d$. Here, $E$ is Young's modulus and $\nu$ is Poisson's ratio.

Plastic yielding is modeled by imposing that the stress remains in a closed convex set $K$ and evolves on its boundary or in its interior. The classical Prandtl--Reuss flow rule \cite{Pra28,Reu30} can be written in variational form \cite{DL76} as
\[
    \sigma(t)\in K,\qquad
    \left(\prz{\eps_p}{t}(t),\tau-\sigma(t)\right)_\Hs\le0\qquad\mbox{for all }\tau\in K,
\]
or equivalently
\[
    S\prz{\sigma}{t} \in \E\left(\prz{u}{t}\right) - \partial I_K(\sigma),
\]
where $I_K$ is the indicator function of $K$ and $\partial I_K$ is its subdifferential in the Hilbert space $\Hs=L^2(\Omega;\Sym_d)$. Here $(A,B)_\Hs := \int_\Omega A:B\,dx$. $K$ is called the constraint set, and the boundary of $K$ is called the yield surface.

When a material that has undergone plastic deformation is loaded again, further plastic deformation often requires a higher stress level than that required for the initial yielding. This phenomenon is called strain hardening or work hardening.
Two standard mechanisms to account for strain hardening are kinematic hardening (a translation of the constraint set in stress space) and isotropic hardening (an expansion or contraction of its size). In this paper, we focus on a linear kinematic hardening model: we assume
\[
    K(t;\alpha)=\tilde{K}(t)+\alpha(t),\qquad 
    \prz{\alpha}{t}=a\,\prz{\eps_p}{t},\qquad 
    \eps_p=\E(u)-S\sigma,
\]
with a scalar hardening parameter $a>0$ and the set
\[
    \tilde{K}(t)=\{\tau\in\Hs:\;|\tau^D|\le g(t,\cdot)\ \text{a.e. in }\Omega\},
\]
where $\tau^D=\tau-\tfrac{\tr\tau}{d}E_d$ is the deviatoric part, $|\cdot|$ is the Frobenius norm, and $g:[0,T]\times\Omega\to [0,\infty)$ prescribes the yield bound. Here, $\alpha$ is called a backstress. This form of the constraint set $\tilde{K}$ corresponds to the well-known von Mises yield criterion.

This model, often referred to as the Melan--Prager model \cite{Mel38,Pra35,Pra49}, is able to capture the Bauschinger effect. In particular, after a material has undergone sufficient plastic deformation and is then unloaded, subsequent loading in the reverse direction may produce new plastic deformation at stress levels lower than those required during the initial loading; this is the Bauschinger effect. 

\subsection{Related work and contribution}

The variational and evolution inclusion perspective adopted in this paper is closely related to Moreau's sweeping process \cite{Mor71,Mor77}. This framework describes an evolution constrained by a time-dependent closed convex set $K(t)\subset\Hs$ in the form
\[
    \sigma'(t)\in -\partial I_{K(t)}(\sigma(t)),
\]
and it is also well suited to variational inequality formulations in plasticity.
In mechanics, variational inequalities were systematically applied to problems such as plasticity, contact, and friction in \cite{DL76}, mainly in quasistatic settings where the stress constraint is prescribed (given by a closed convex set).

From a broader viewpoint, evolution inclusions governed by subdifferentials provide a common analytic framework for hysteresis and plasticity. 
In particular, \cite{Vis06} organizes several classical hysteresis operators, including Prandtl-type and play/stop operators, and their representations by variational inequalities.
Moreover, \cite{Sho97} presents an abstract theory of evolution problems driven by monotone operators, including subdifferentials of indicator functions, together with applications to PDEs.

Regarding the coupled system with the equation of motion and the perfect plasticity model (the case where $a=0$), well-posedness of the system was proved by using the theory of maximal monotone operators and the Yosida regularization in \cite{KC20}.
In addition, \cite{AFK23} studied a model with a time-dependent constraint set and used abstract results for evolution inclusions with time-dependent constraints.
In the case where the time-dependent constraint set is given as an external datum, a projection-based approach was proposed and analyzed in \cite{AM25}. In \cite{AM25}, no spatial smoothness of the threshold function is assumed, nor strictly positive lower bound of the form $0<c\le g$ is imposed.

In contrast to the case where the constraint set is prescribed, in the linear kinematic hardening model considered in this paper, the admissible set is translated by the backstress $\alpha$. Hence the constraint set $K(t;\alpha)$ depends on the unknown $\alpha$, and the system has a quasi-variational structure. 
For the same type of linear kinematic hardening model, well-posedness was obtained in \cite{AFK26} by an evolution-equation approach.

\medskip
\noindent\textbf{Contribution of this paper.}
In this paper, we prove existence and uniqueness of a weak solution to the linear kinematic hardening system of the same type as in \cite{AFK26} by using a projection-based approach.
Compared with \cite{AFK26}, we weaken the assumptions on the yield bound $g$ and do not assume spatial smoothness or a strictly positive lower bound of the form $0<c\le g$.
Furthermore, the resulting structure is consistent with spatial discretization: by replacing the function spaces with finite element spaces, the same structure gives a direct way to construct numerical schemes.

\subsection{Problem}

We assume the existence of two subsets $\Gamma_1$ and $\Gamma_2$ of the boundary $\Gamma\defeq\partial\Omega$
\[
    \Gamma=\Gamma_1\cup\Gamma_2, \qquad
    \Gamma_1\cap\Gamma_2=\emptyset, \qquad
    \mathcal{H}^{d-1}(\Gamma_1)>0.
\]
Here, $\mathcal{H}^{d-1}(\Gamma_1)$ denotes the ($d-1$)-dimensional Hausdorff measure of $\Gamma_1$.
We couple the equation of motion with a Kelvin--Voigt type viscosity. Denoting $v=\prz{u}{t}$ and $\xi=\prz{\eps_p}{t}$, the unknowns are the velocity $v$, the stress $\sigma$, and the internal variable $\alpha$, and the system reads
\begin{align*}
\left\{\begin{aligned}
    \rho\prz{v}{t} &= \eta\dvg\E(v) + \dvg\sigma + F &&\mbox{in }(0,T)\times\Omega,\\
    S\prz{\sigma}{t} &\in \E(v) - \partial I_K(\sigma) &&\mbox{in }(0,T)\times\Omega,\\
    \prz{\alpha}{t} &= a\xi
    &&\mbox{in }(0,T)\times\Omega,\\
    \xi &= \E(v) - S\prz{\sigma}{t}
    &&\mbox{in }(0,T)\times\Omega,\\
    v &=v_{\rm b} &&\mbox{on }(0,T)\times\Gamma_1,\\
    (\eta\E(v) + \sigma) n &=t_{\rm b} &&\mbox{on }(0,T)\times\Gamma_2,\\
    v(0,\cdot) &= v_0 &&\mbox{in }\Omega,\\
    \sigma(0,\cdot) &= \sigma_0 &&\mbox{in }\Omega,\\
    \alpha(0,\cdot) &= \alpha_0 &&\mbox{in }\Omega,\\
\end{aligned}\right.
\end{align*}
where 
\begin{align*}\begin{array}{c}
    {\DS K(t;\alpha) \defeq \tilde{K}(t) + \alpha(t),\quad 
    \tilde{K}(t) \defeq \left\{\tau\in\Hs : |\tau^D|\le g(t,\cdot)\mbox{ a.e. in }\Omega\right\}.}
\end{array}\end{align*}
Here, 
the density $\rho>0$,
the viscosity coefficient $\eta>0$,
the external force $F:(0,T)\times\Omega\rightarrow\R^d$,
the coefficient $a > 0$,
the elastic compliance tensor $S$ with \eqref{def:S},
boundary values $v_{\rm b}:(0,T)\times\Gamma_1\rightarrow\R^d$ and 
$t_{\rm b}:(0,T)\times\Gamma_2\rightarrow\R^d$, and
initial values $v_0:\Omega\rightarrow \R^d, \sigma_0, \alpha_0:\Omega\rightarrow \Sym_d$ are given with $\sigma_0 \in K(0;\alpha_0)$. Also, $n$ is the outward unit normal vector to the boundary $\Gamma$.

In subsequent sections we set $\rho=1$ for simplicity and introduce notation used throughout the analysis.
We assume that there exists $\tilde{v}_{\rm b}:[0,T]\times\ov{\Omega}\rightarrow\R^d$ such that $\tilde{v}_{\rm b} = v_{\rm b}$ on $(0,T)\times\Gamma_1$. Then, by replacing
    $v$ with $v-\tilde{v}_{\rm b}$,
    $F$ with $F - \rho\prz{\tilde{v}_{\rm b}}{t} + \eta\dvg\E(\tilde{v}_{\rm b})$,
    $t_{\rm b}$ with $t_{\rm b} - \eta\E(\tilde{v}_{\rm b})n$, and
    $v_0$ with $v_0 - \tilde{v}_{\rm b}(0,\cdot)$,
and setting $h=\E(\tilde{v}_{\rm b})$, the problem can be reformulated as follows.

\begin{problem}\label{Prob}
    Find 
        $v:[0,T]\times\Omega\rightarrow\R^d$,
        $\sigma:[0,T]\times\Omega\rightarrow\Sym_d$, and
        $\alpha:[0,T]\times\Omega\rightarrow\Sym_d$ 
    such that
    \begin{align*}
    \left\{\begin{aligned}
        \prz{v}{t} &= \eta\dvg\E(v) + \dvg\sigma + F &&\mbox{in }(0,T)\times\Omega,\\
        S\prz{\sigma}{t} &\in \E(v) + h - \partial I_K(\sigma) &&\mbox{in }(0,T)\times\Omega,\\
        \prz{\alpha}{t} &= a\xi
        &&\mbox{in }(0,T)\times\Omega,\\
        \xi &= \E(v)+h - S\prz{\sigma}{t}
        &&\mbox{in }(0,T)\times\Omega,\\
        v &=0 &&\mbox{on }(0,T)\times\Gamma_1,\\
        (\eta\E(v) + \sigma) n &=t_{\rm b} &&\mbox{on }(0,T)\times\Gamma_2,\\
        v(0,\cdot) &= v_0 &&\mbox{in }\Omega,\\
        \sigma(0,\cdot) &= \sigma_0 &&\mbox{in }\Omega,\\
        \alpha(0,\cdot) &= \alpha_0 &&\mbox{in }\Omega,\\
    \end{aligned}\right.
    \end{align*}
    where 
    \begin{align*}\begin{array}{c}
        {\DS K(t;\alpha) \defeq \tilde{K}(t) + \alpha(t),\quad 
        \tilde{K}(t) \defeq \left\{\tau\in\Hs : |\tau^D|\le g(t)\mbox{ a.e. in }\Omega\right\}.}
    \end{array}\end{align*}
\end{problem}

\subsection{Notations}

For a Banach space $X$, the dual pairing between $X$ and the dual space $X^*$ is denoted by $\langle\cdot,\cdot\rangle_{\!X^*,X}$.
We say that a function $u:[0, T]\rightarrow X$ is weakly continuous if, for all $f \in X^*$, the function defined by $[0,T]\ni t\mapsto\langle f,u(t)\rangle_{\!X^*,X}\in\R$ is continuous. We denote by $C^0([0,T];X_w)$ the set of functions defined on $[0,T]$ with values in $X$ which are weakly continuous.
Let $\dt=T/N$ ($N\in\N$). For two sequences $(x_k)_{k=0}^N$ and $(y_k)_{k=1}^N$ in $X$, we define a piecewise linear interpolant $\hat{x}_\dt \in W^{1, \infty}(0, T; X)$ of $(x_k)_{k=0}^N$ and a piecewise constant interpolant $\bar{y}_\dt \in L^\infty(0, T; X)$ of $(y_k)_{k=1}^N$, respectively, by
\[\begin{aligned}
    \hat{x}_\dt(t)&:=x_{k-1}+\frac{t-t_{k-1}}{\dt}(x_k-x_{k-1})
    &&\text{for }t\in[t_{k-1},t_k] \mbox{ and }k=1,2,\ldots,N,\\
    \bar{y}_\dt(t)&:=y_k
    &&\text{for }t\in(t_{k-1},t_k] \mbox{ and }k=1,2,\ldots,N,
\end{aligned}\]
where $t_k=k\dt$.
We define a backward difference operator by 
\[
    D_\dt x_k:=\frac{x_k-x_{k-1}}{\dt},\qquad
    D_\dt y_l:=\frac{y_l-y_{l-1}}{\dt}
\]
for $k=1,2,\ldots,N$ and $l=2,3,\ldots,N$.
Then, $\frac{\partial\hat{x}_\dt}{\partial t}=D_\dt x_k$
on $(t_{k-1},t_k)$ for all $k=1,2,\ldots,N$.

We define the function spaces
$\Hv \defeq L^2(\Omega;\R^d)$,
$\Hs \defeq L^2(\Omega;\Sym_d)$,
$\Vv \defeq \{\varphi\in H^1(\Omega;\R^d):\varphi=0\mbox{ on }\Gamma_1\}$,
and let $\Vv^*$ be the dual space of $\Vv$. 
We also define the elasticity tensor $C:\Sym_d\to\Sym_d$ by
\[
    C\eps=\frac{E}{1+\nu}\left(\eps+\frac{\nu}{1-(d-1)\nu}(\tr\eps)E_d\right),
\]
so that $C$ is the inverse of $S$, i.e., $CS\tau=SC\tau=\tau$ for all $\tau\in\Sym_d$. 
Moreover, we introduce the norms $\|\cdot\|_S:\Hs\rightarrow\R$ and $\|\cdot\|_C:\Hs\rightarrow\R$ as follows:
\[
    \|\tau\|_S^2 \defeq \into (S\tau):\tau dx,\qquad 
    \|\tau\|_C^2 \defeq \into (C\tau):\tau dx\qquad\text{for all }\tau\in\Hs,
\]
where $A:B=\sum_{i,j=1}^{d} A_{ij}B_{ij}$. These are norms on $\Hs$ by Lemma \ref{Lem:posSC}.

\subsection{Definition of a solution to Problem \ref{Prob}}

In the weak formulation of Problem \ref{Prob}, we collect the external force and boundary traction into a single functional $f\in L^2(0,T;\Vv^*)$ defined by
\begin{align*}
    \blaket{f(t)}{\varphi}
    \defeq \into F(t)\cdot\varphi dx + \int_{\Gamma_2} t_{\rm b}(t)\cdot\varphi ds
\end{align*}
for all $\varphi\in\Vv$ and a.e. $t\in(0,T)$.
The solution to Problem \ref{Prob} is defined as follows.

\begin{definition}
    Given $\eta>0$, $a > 0$, $v_0\in\Hv$, $\sigma_0\in\Hs$, $\alpha_0\in\Hs$, $f\in L^2(0,T;\Vv^*)$, $S:\Sym_d\rightarrow\Sym_d$, $h\in L^2(0,T;\Hs)$, $g\in H^1(0,T;\Lo)$. Assume that $S$ satisfies \eqref{def:S} and that $g(t,x)\ge0$  for a.e. $(t,x)\in(0,T)\times\Omega$.
    We call the triplet $(v, \sigma, \alpha)\in (H^1(0,T;\Vv^*)\cap L^2(0,T;\Vv)) \times$ $H^1(0,T;\Hs) \times H^1(0,T;\Hs)$ a solution to Problem \ref{Prob} if:
    $v(0)=v_0$, $\sigma(0)=\sigma_0$, $\alpha(0)=\alpha_0$ and for all $t\in[0,T]$,
    \[
        \sigma(t) \in K(t;\alpha) \defeq \tilde{K}(t) + \alpha(t)
    \]
    is satisfied, and for a.e. $t\in(0,T)$,
    \begin{align}\label{sys:Prob}\left\{\begin{aligned}
        &\blaket{v'}{\varphi}
        + \eta (\E(v),\E(\varphi))_\Hs + (\sigma,\E(\varphi))_\Hs
        = \blaket{f}{\varphi} 
        &&\mbox{for all }\varphi\in\Vv,\\
        &\left(S\sigma'-\E(v),\sigma-\tau\right)_\Hs 
        \le \left(h,\sigma-\tau\right)_\Hs
        &&\mbox{for all }\tau\in K(t;\alpha),\\
        &\alpha' = a\xi
        &&\mbox{in }\Hs
    \end{aligned}\right.
    \end{align}
    holds. Here, $\xi = \E(v) + h - S\sigma'$.
\end{definition}

\section{Proposed scheme}

Let 
\[
    f_n \defeq \frac{1}{\dt}\int_{t_{n-1}}^{t_n} f(t)dt, \qquad
    h_n \defeq \frac{1}{\dt}\int_{t_{n-1}}^{t_n} h(t)dt, \qquad
    g_n \defeq g(t_n), 
\]
for all $n=1,2,\ldots,N$.

We consider the following numerical scheme: 
\begin{itembox}[l]{Proposed scheme (P)} 
For each $n=1,2,\ldots,N$, find $v_n\in\Vv$, $\sigmas{n},\sigma_n,\alpha_n\in\Hs$ satisfying
\begin{align*}\left\{\begin{aligned}
    &\inner{\frac{v_n-v_{n-1}}{\dt}}{\varphi}{\Hv}
    + \eta (\E(v_n),\E(\varphi))_\Hs + (\sigmas{n},\E(\varphi))_\Hs
    = \blaket{f_n}{\varphi} && \mbox{for all }\varphi\in\Vv,\\
    &S\frac{\sigmas{n}-\sigma_{n-1}}{\dt} - \E(v_n) = h_n
    &&\mbox{in }\Hs,\\
    & \frac{\alpha_n - \alpha_{n-1}}{\dt}
    = a\xi_n &&\mbox{in }\Hs,\\
    &\sigma_n - \alpha_n
    = \Pop{g_n}\left(\sigmas{n} - \alpha_n \right)
    &&\mbox{in }\Hs.\\
\end{aligned}\right.\end{align*}
where 
\[
    \xi_n \defeq \E(v_n) + h_n - S\frac{\sigma_n - \sigma_{n-1}}{\dt},
\]
and $\Pop{R}:\Sym_d\rightarrow\Sym_d$ is defined for $A\in\Sym_d$, 
\begin{align*}
    \Pop{R}(A) &\defeq \left\{\begin{aligned}
        &A
        &&\mbox{if }|A^D|\le R,\\
        &\frac{\tr A}{d}E_d + R\frac{A^D}{|A^D|}
        &&\mbox{if }|A^D|>R.
    \end{aligned}\right.
\end{align*}
\end{itembox}

Scheme (P) is a projection-based time discretization in which each step is split into a linear viscous-elastic subproblem and a projection onto the translated constraint set. More precisely, for a fixed $n\in\{1,\dots,N\}$ we proceed as follows.

\smallskip
\noindent\emph{(i) Trial step.}
The first and second equations form a linear system for $(v_n,\sigmas{n})$.
The quantity $\sigmas{n}$ plays the role of a trial stress obtained by ignoring the constraint $\sigma\in K(t;\alpha)$ at this stage.

\smallskip
\noindent\emph{(ii) Backstress update.}
The backstress $\alpha_n$ is updated by the third equation using $\xi_n$ as an approximation of the plastic strain rate $\prz{\eps_p}{t}(t_n)$.

\smallskip
\noindent\emph{(iii) Projection step.}
The fourth equation enforces the stress constraint at $t_n$ by projecting the shifted trial stress $\sigmas{n}-\alpha_n$ onto the set $\tilde K(t_n)$ (here realized through the pointwise projection operator $\Pop{g_n}$).

The correction $\sigmas{n}-\sigma_n$ generated by the projection provides the discrete plastic flow. Indeed, combining the second equation with the definition of $\xi_n$ yields
\begin{align}\label{eq:xiequiv}
    \xi_n = S\frac{\sigmas{n} - \sigma_n}{\dt}.
\end{align}
Note that $\xi_n$ depends on $\sigma_n$, while $\sigma_n$ is defined through a projection involving $\alpha_n$, and $\alpha_n$ itself is updated by $\xi_n$. Hence, at each step the variables $(\sigma_n,\alpha_n)$ are coupled through a nonlinear relation, and solvability of (P) is not immediate. We address this point in the next section.

\section{Solvability of the Proposed Scheme}

We verify that the proposed scheme (P) can be solved at each step. From the second equation of (P) and the definition of $C$, we have 
\begin{align}\label{eq:equiv2nd}
    \sigmas{n}=\sigma_{n-1}+\dt C(\E(v_n)+h_n).
\end{align} 
From the first equation of (P), it holds that
\begin{align*}
    &\inner{\frac{v_n-v_{n-1}}{\dt}}{\varphi}{\Hv}
	+ \left((\eta E_d+\dt C)\E(v_n),\E(\varphi)\right)_\Hs\\
	=~& \blaket{f_n}{\varphi} 
	- \left(\sigma_{n-1}+\dt Ch_n,\E(\varphi)\right)_\Hs.
\end{align*}
If $v_{n-1}\in\Vv$, $\sigma_{n-1}\in\Hs$ are known, we uniquely determine $v_n\in\Vv$ by the Korn inequality (Lemma \ref{Lem:Korn}), the positive definiteness of $C$ (Lemma \ref{Lem:posSC}), and the Lax--Milgram theorem. We then obtain $\sigmas{n}$ from \eqref{eq:equiv2nd}.

Next, we calculate $\alpha_n$ and $\sigma_n$ from the third and fourth equations of (P):
\begin{align}\label{sys:3rd4th}\left\{\begin{aligned}
    & \alpha_n - \alpha_{n-1}
    = aS(\sigmas{n} - \sigma_n) &&\mbox{in }\Hs,\\
    &\sigma_n - \alpha_n
    = \Pop{g_n}\left(\sigmas{n} - \alpha_n \right)
    &&\mbox{in }\Hs,\\
\end{aligned}\right.\end{align}
where we used \eqref{eq:xiequiv}.
From the definition of $\Pop{g_n}$ and \eqref{sys:3rd4th}, we know that 
\begin{align}\label{sys:tr}
    \tr\sigma_n=\tr\sigmas{n},\qquad
    \tr\alpha_n=\tr\alpha_{n-1}.
\end{align}
Since $(S\tau)^D=\frac{1+\nu}{E}\tau^D$ and $(\Pop{g_n}(\tau))^D = \Pop{g_n}(\tau^D)$, the deviatoric parts of \eqref{sys:3rd4th} are written as
\begin{align}\label{sys:alphasigma}\left\{\begin{aligned}
    & \alpha^D_n - \alpha^D_{n-1}
    = b((\sigmas{n})^D - \sigma^D_n) &&\mbox{in }\Hs,\\
    &\sigma^D_n - \alpha^D_n
    = \Pop{g_n}\left((\sigmas{n})^D - \alpha^D_n \right)
    &&\mbox{in }\Hs,\\
\end{aligned}\right.\end{align}
where $b = a\frac{1+\nu}{E}$.
The system \eqref{sys:alphasigma} determines $\alpha_n^D$ and $\sigma_n^D$ from $(\sigmas{n})^D$ and $\alpha_{n-1}^D$.
Moreover, it can be solved explicitly, pointwise (a.e.) in $\Omega$, as we show below.
To simplify the notation, we omit the space variable $(x)$ and write
$\sigmas{n}$, $\alpha_{n-1}$, $g_n$, $\sigma_n$, and $\alpha_n$
instead of $\sigmas{n}(x)$, $\alpha_{n-1}(x)$, $g_n(x)$, $\sigma_n(x)$, and $\alpha_n(x)$, respectively.

\smallskip\noindent
(i) If $|(\sigmas{n}-\alpha_{n-1})^D| \le g_n$, then
\[
    \alpha^D_n = \alpha^D_{n-1},\qquad
    \sigma^D_n = (\sigmas{n})^D.
\]
This corresponds to the case where no plastic deformation occurs.

\smallskip\noindent
(ii) Let us consider the case when $|(\sigmas{n}-\alpha_{n-1})^D| > g_n$. From the second equation of \eqref{sys:alphasigma}, $(\sigmas{n})^D$, $\sigma^D_n$, and $\alpha^D_n$ are collinear. From the first equation of \eqref{sys:alphasigma}, $\alpha^D_{n-1}$ also lies on the line passing through $(\sigmas{n})^D$, $\sigma^D_n$, and $\alpha^D_n$ (see Figure \ref{fig}).

\begin{figure}[htp]\begin{center}
\begin{tikzpicture}[>=stealth]

  \draw (-3,0) -- (5,0);

  \fill (4,0) circle (2pt) node[above] {$(\sigma_n^*)^D$};
  \fill (3,0) circle (2pt) node[above] {$\sigma^D_n$};
  \fill (0,0) circle (2pt) node[above] {$\alpha^D_n$};
  \fill (-2,0) circle (2pt) node[above] {$\alpha^D_{n-1}$};

\end{tikzpicture}
\caption{The geometric configuration of $\alpha_{n-1}^D$, $\alpha_n^D$, $\sigma_n^D$, and $(\sigma_n^*)^D$ in $\Sym_d$. Points $\alpha_{n-1}^D, \alpha_n^D, \sigma_n^D, (\sigma_n^*)^D$ lie on a line in this order. The distances satisfy
$|\sigma_n^D - \alpha_n^D| = g_n$ and $|\alpha_n^D - \alpha_{n-1}^D| = b|(\sigma_n^*)^D - \sigma_n^D|$.}\label{fig}
\end{center}\end{figure}
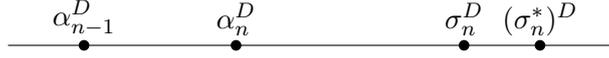
Hence, we can find $\alpha^D_n$ and $\sigma^D_n$:
\begin{align*}
    \alpha^D_n - \alpha^D_{n-1} 
    = \frac{b}{b+1}\left((\sigmas{n})^D - \alpha^D_{n-1} 
    - g_n\frac{(\sigmas{n})^D - \alpha^D_{n-1}}{|(\sigmas{n})^D - \alpha^D_{n-1}|}\right),\quad
    \sigma^D_n - (\sigmas{n})^D
    = - \frac{1}{b}(\alpha^D_n - \alpha^D_{n-1}).
\end{align*}

Therefore, we can summarize \eqref{sys:tr} and the cases (i) and (ii) as follows: $\alpha_n$ and $\sigma_n$ can be expressed as
\begin{align}\label{sol:alphasigma}
    \alpha_n - \alpha_{n-1} = \frac{b}{b+1}\left(\sigmas{n} - \alpha_{n-1} - \Pop{g_n}(\sigmas{n} - \alpha_{n-1})\right),
    \quad
    \sigma_n - \sigmas{n}
    = - \frac{1}{b}(\alpha_n - \alpha_{n-1}).
\end{align}
If $\sigmas{n}$ and $\alpha_{n-1}$ are known, then $\alpha_n$ and $\sigma_n$ are determined by \eqref{sol:alphasigma}.

\section{Main results}

The discrete solutions produced by the proposed scheme are stable in the following sense.

\begin{Th}\label{Th:stab}~
    \begin{enumerate}
        \item[(i)] There exists a constant $c_1 > 0$ independent of $\dt$ such that for all $0<\dt\le1$,
        \[\begin{aligned}
            &\|\hat{v}'_\dt\|_{L^2(0,T;\Vv^*)}^2
            +\|\bar{v}_\dt\|_{L^\infty(0,T;\Hv)}^2 
            + \|\bar{v}_\dt\|_{L^2(0,T;\Vv)}^2
            + \frac{1}{\dt}\|\hat{v}_\dt - \bar{v}_\dt\|_{L^2(0,T;\Hv)}^2\\
            &+ \|\bar{\sigma}^*_\dt\|_{L^\infty(0,T;\Hs)}^2 
            + \|\bar{\sigma}_\dt\|_{L^\infty(0,T;\Hs)}^2
            + \frac{1}{\dt}\|\hat{\sigma}_\dt - \bar{\sigma}_\dt\|_{L^2(0,T;\Hs)}^2\\
            &+ \frac{1}{\dt}\|\bar{\sigma}^*_\dt - \bar{\sigma}_\dt\|_{L^2(0,T;\Hs)}^2
            + \|\bar{\alpha}_\dt\|_{L^\infty(0,T;\Hs)}^2
            + \frac{1}{\dt}\|\hat{\alpha}_\dt - \bar{\alpha}_\dt\|_{L^2(0,T;\Hs)}^2\\
            \le~& c_1(\|v_0\|_\Hv^2 + \|\sigma_0\|_\Hs^2 + \|\alpha_0\|_\Hs^2
            + \|f\|_{L^2(0,T;\Vv^*)}^2 + \|h\|_{L^2(0,T;\Hs)}^2).
        \end{aligned}\]
        In particular, it holds that
        $\|\bar{\sigma}^*_\dt - \bar{\sigma}_\dt\|_{L^2(0,T;\Hs)}\rightarrow0$
        as $\dt\rightarrow0$.
        \item[(ii)] There exists a constant $c_2 > 0$ independent of $\dt$ such that for all $0<\dt\le1$,
        \[\begin{aligned}
            & \|\hat{\sigma}_\dt\|_{H^1(0,T;\Hs)}^2 
            + \|\hat{\alpha}_\dt\|_{H^1(0,T;\Hs)}^2 
            + \|\bar{\xi}_\dt\|_{L^2(0,T;\Hs)}^2\\
            \le~& c_2(\|v_0\|_\Hv^2 + \|\sigma_0\|_\Hs^2 + \|\alpha_0\|_\Hs^2 
            + \|f\|_{L^2(0,T;\Vv^*)}^2 + \|h\|_{L^2(0,T;\Hs)}^2 + \|g\|_{H^1(0,T;\Lo)}^2).
        \end{aligned}\]
    \end{enumerate}
\end{Th}

From the boundedness obtained in Theorem \ref{Th:stab}, we obtain that the sequences $(\hat{v}_\dt)_{0<\dt<1}$, $(\bar{v}_\dt)_{0<\dt<1}$, $(\bar{\sigma}^*_\dt)_{0<\dt<1}$, $(\bar{\sigma}_\dt)_{0<\dt<1}$, $(\hat{\sigma}_\dt)_{0<\dt<1}$, $(\bar{\alpha}_\dt)_{0<\dt<1}$, and $(\hat{\alpha}_\dt)_{0<\dt<1}$ have subsequences that converge weakly. Since the resulting limit $(v,\sigma,\alpha)$ is a solution to Problem \ref{Prob}, the following theorem can be established.

\begin{Th}\label{Th:exist}
    There exists a unique solution $(v,\sigma,\alpha)$ to Problem \ref{Prob}.
\end{Th}

\section{Proofs}

\subsection{Preliminary result}

We recall the Korn inequality and the discrete Gronwall inequality.

\begin{Lem}[{\cite[Lemma 5.4.18]{Trangenstein13}}]\label{Lem:Korn}
    There exists a constant $c_K>0$ such that 
    \[
        \frac{1}{c_K}\|\varphi\|_\Vv \le \|\E(\varphi)\|_\Hs 
        \le c_K\|\varphi\|_\Vv \quad\mbox{for all }\varphi\in \Vv.
    \]
\end{Lem}

\begin{Lem}[{\cite[Lemma 5.1]{HR90}}]\label{lem:gronwall}
  Let $\dt,\beta>0$ and let non-negative sequences 
  $(a_k)^N_{k=0}$, $(b_k)^N_{k=0}$, $(c_k)^N_{k=0}$, 
  $(\alpha_k)^N_{k=0} \subset \{x\in\R : x \ge 0\}$ satisfy that
  \[
      a_m+\dt\sum^m_{k=0}b_k
      \le \dt\sum^m_{k=0}\alpha_k a_k+\dt\sum^m_{k=0}c_k+\beta
      \qquad\mbox{for all }m=0,1,\ldots,N.
  \]
  If $\alpha_k\dt<1$ for all $k=0,1,\ldots,N$, then we have
  \[
      a_m+\dt\sum^m_{k=0}b_k
      \le e^C\left(\dt\sum^m_{k=0}c_k+\beta\right)
      \qquad\mbox{for all }m=0,1,\ldots,N,
  \]
  where $C:=\dt\sum^N_{k=0}\frac{\alpha_k}{1-\alpha_k\dt}$.
\end{Lem}

We prepare the following lemma, which is derived using the Ascoli--Arzel{\`{a}} theorem.

\begin{Lem}[{\cite[Lemma 5.3]{AM25}}]\label{Lem:AA}
    Let $X$ be a Banach space such that $X^*$ is separable. If the sequence $(u_n)_{n\in\N}\subset C([0,T];X)$ satisfies the following two conditions:
    \begin{enumerate}
        \item[(i)] there exists a constant $c>0$ such that for all $n\in\N$ and $t\in[0,T]$, $\|u_n(t)\|_X < c$,
        \item[(ii)] $(u_n)$ is equicontinuous, i.e., for all $t\in[0,T]$ and $\eps>0$, there exists $\delta>0$ such that if $s\in[0,T]$ satisfies $|s-t|<\delta$, then $\|u_n(s) - u_n(t)\|_X < \eps$ for all $n\in\N$,
    \end{enumerate}
    then there exist a subsequence $(n_k)_{k\in\N}$ and $u\in C([0,T];X_w)$ such that for all $t\in[0,T]$,
    \[
        u_{n_k}(t) \rightharpoonup u(t) \mbox{ weakly in } X
    \]
    as $k\rightarrow\infty$.
\end{Lem}

We show some properties of the function $\Pop{R}:\Sym_d\rightarrow\Sym_d$.

\begin{Prop}[{\cite[Proposition 5.4]{AM25}}]\label{Prop:Phi}
    \begin{enumerate}
        \item[(i)] It holds that for all $R_1, R_2\ge0$ and $A \in \Sym_d$,
        \[
            |\Pop{R_1}(A) - \Pop{R_2}(A)| \le |R_1 - R_2|.
        \]
        \item[(ii)] It holds that for all $R\ge0$ and $A, B \in \Sym_d$ with $|B^D|\le R$,
        \[
            (\Pop{R}(A) - A, \Pop{R}(A) - B) \le 0.
        \]
        \item[(iii)] It holds that for all $R\ge0$ and $A, B \in \Sym_d$,
        \[
            |\Pop{R}(A) - \Pop{R}(B)| \le |A - B|.
        \]
    \end{enumerate}
\end{Prop}

The following two lemmas provide key inequalities for handling the compliance tensor $S$.

\begin{Lem}[The positive definiteness and continuity of $S$]\label{Lem:posSC}
    There exist two constants $c_S, c_C > 0$ such that 
    \begin{align}\label{ineq:posSC}
        \frac{1}{c_S}\|\tau\|_\Hs^2 \le \|\tau\|_S^2 \le c_S\|\tau\|_\Hs^2,\qquad
        \frac{1}{c_C}\|\tau\|_\Hs^2 \le \|\tau\|_C^2 \le c_C\|\tau\|_\Hs^2
        \qquad\mbox{for all }\tau\in\Hs.
    \end{align}
\end{Lem}

\begin{proof}
    It holds that for all $\tau\in\Hs$, 
    \begin{align*}
        \|\tau\|_S^2
        &= \into (S\tau):\tau dx 
        = \into \left(\frac{1+\nu}{E}|\tau|^2 - \frac{\nu}{E}(\tr \tau)^2\right) dx
        = \into \left(\frac{1+\nu}{E}|\tau^D|^2 + \frac{1-(d-1)\nu}{E}\frac{(\tr \tau)^2}{d}\right) dx\\
        &\ge \frac{\min\{1+\nu,1-(d-1)\nu\}}{E}\into \left(|\tau^D|^2 + \frac{(\tr \tau)^2}{d}\right) dx
        = \frac{\min\{1+\nu,1-(d-1)\nu\}}{E}\|\tau\|_\Hs^2,\\
        \|\tau\|_S^2
        &\le \frac{\max\{1+\nu,1-(d-1)\nu\}}{E}\|\tau\|_\Hs^2,
    \end{align*}
    where we have used $|A|^2 = |A^D|^2 + \frac{1}{d}(\tr A)^2$ for all $A\in\Sym_d$.
    It follows that $1+\nu > 0$ and $1-(d-1)\nu > 0$ since $-1<\nu<\frac{1}{d-1}$. Hence, if we set $c_S\defeq\max\{\frac{1+\nu}{E}, \frac{1-(d-1)\nu}{E}, \frac{E}{1+\nu}, \frac{E}{1-(d-1)\nu}\}$, then the first inequality of \eqref{ineq:posSC} holds.

    On the other hand, it holds that for all $\tau\in\Hs$,
    \begin{align*}
        \|\tau\|_C^2
        &= \into (C\tau):\tau dx 
        = \frac{E}{1+\nu}\into \left(|\tau|^2 + \frac{\nu}{1-(d-1)\nu}(\tr \tau)^2\right) dx \\ 
        &= \into \left(\frac{E}{1+\nu}|\tau^D|^2 + \frac{E}{1-(d-1)\nu}\frac{(\tr \tau)^2}{d}\right) dx 
        \ge \min\left\{\frac{E}{1+\nu},\frac{E}{1-(d-1)\nu}\right\}\|\tau\|_\Hs^2,\\
        \|\tau\|_C^2
        &\le \max\left\{\frac{E}{1+\nu},\frac{E}{1-(d-1)\nu}\right\}\|\tau\|_\Hs^2.
    \end{align*}
    Hence, if we set $c_C\defeq c_S$, then the second inequality of \eqref{ineq:posSC} holds.
\end{proof}

\begin{Lem}\label{Lem:Sineq}
    It holds that for all $n=1,2,\ldots,N$,
	\[
		(\xi_n, \tau - (\sigma_n-\alpha_n))_\Hs \le 0
		\qquad\mbox{for all }\tau\in\Hs\mbox{ with }|\tau^D|\le g_n.
	\]
\end{Lem}

\begin{proof}
From the 4th equation of (P), $\sigma_n - \alpha_n = \Pop{g_n}(\sigmas{n} - \alpha_n )$, we have 
\begin{align*}
    \left(\sigma_n - \sigmas{n},
    (\sigma_n - \alpha_n) - \tilde{\tau} \right)_\Hs \le 0
    \quad\mbox{for all }\tilde{\tau}\in\Hs\mbox{ with }|\tilde{\tau}^D| \le g_n.
\end{align*}
Noting that $(S\tau)^D = \frac{1+\nu}{E}\tau^D$, for any $\tau\in\Hs$,
\begin{align*}
    &\left|\left((\sigma_n - \alpha_n) - \frac{E}{1+\nu}S\left((\sigma_n - \alpha_n) - \tau\right)\right)^D\right|
    = \left|(\sigma_n - \alpha_n)^D - \left((\sigma_n - \alpha_n)^D - \tau^D\right)\right|
    = |\tau^D|.
\end{align*}
Hence, by setting $\tilde{\tau} \defeq (\sigma_n - \alpha_n) - \frac{E}{1+\nu}S\left((\sigma_n - \alpha_n) - \tau\right)~(\tau\in\Hs\mbox{ with }|\tau^D| \le g_n)$, 
\begin{align*}
    \left(\sigma_n - \sigmas{n},
    S\left((\sigma_n - \alpha_n) - \tau\right) \right)_\Hs
    = \left(S(\sigma_n - \sigmas{n}),
    (\sigma_n - \alpha_n) - \tau \right)_\Hs \le 0.
\end{align*}
Since $\xi_n = S\frac{\sigmas{n}-\sigma_n}{\dt}$,
\begin{align*}
    \left(\xi_n, \tau - (\sigma_n - \alpha_n) \right)_\Hs \le 0.
\end{align*}
\end{proof}

\subsection{Stability}

In this subsection, we prove Theorem \ref{Th:stab}.

First, we show (i). By substituting $\varphi=v_n$ into the first equation of (P) and using Korn's inequality, we have
\begin{align*}
	&\frac{1}{2\dt}\left(\|v_n\|_\Hv^2 - \|v_{n-1}\|_\Hv^2
	+ \|v_n - v_{n-1}\|_\Hv^2\right)
	+ \eta \|\E(v_n)\|_\Hs^2 + (\sigmas{n},\E(v_n))_\Hs\\
	=~& \blaket{f_n}{v_n} 
	\le c_K\|f_n\|_{\Vv^*}\|\E(v_n)\|_\Hs
	\le \frac{\eta}{2}\|\E(v_n)\|_\Hs^2 + \frac{c_K^2}{2\eta}\|f_n\|_{\Vv^*}^2.
\end{align*}
On the other hand, taking the $\Hs$-inner product of the second equation of (P) with $\sigmas{n}$, we obtain
\begin{align*}
	&\frac{1}{2\dt}\left(\|\sigmas{n}\|_S^2 - \|\sigma_{n-1}\|_S^2
	+ \|\sigmas{n} - \sigma_{n-1}\|_S^2\right)
	- (\E(v_n),\sigmas{n})_{\Hs}\\
    =~ &(h_n, \sigmas{n})_{\Hs}
    \le \|h_n\|_\Hs \|\sigmas{n}\|_\Hs
    \le c_S^{1/2} \|h_n\|_\Hs \|\sigmas{n}\|_S
	\le c_S \|h_n\|_\Hs^2 + \frac{1}{4}\|\sigmas{n}\|_S^2.
\end{align*}
By adding the above two inequalities, we obtain
\begin{align}\label{ineq:stab1}\begin{aligned}
	&\frac{1}{\dt}\left(\|v_n\|_\Hv^2 - \|v_{n-1}\|_\Hv^2
	+ \|\sigmas{n}\|_S^2 - \|\sigma_{n-1}\|_S^2\right)
    + \eta \|\E(v_n)\|_\Hs^2 
	+ \frac{1}{\dt}\|v_n - v_{n-1}\|_\Hv^2 + \frac{1}{\dt}\|\sigmas{n} - \sigma_{n-1}\|_S^2\\
	\le~& \frac{c_K^2}{\eta}\|f_n\|_{\Vv^*}^2 + 2c_S\|h_n\|_\Hs^2 + \frac{1}{2}\|\sigmas{n}\|_S^2.
\end{aligned}\end{align}
Putting $\tau \defeq 0$ in Lemma \ref{Lem:Sineq}, from the definition of $\xi_n$ and the fourth equation of (P), we have
\begin{align}\label{ineq:stab2}\begin{aligned}
	&0 \ge (\xi_n, -\sigma_n + \alpha_n)
	= \left(S\frac{\sigmas{n}-\sigma_n}{\dt}, - \sigma_n \right)_\Hs
  	+ \frac{1}{a}\left(\frac{\alpha_n - \alpha_{n-1}}{\dt}, \alpha_n \right)_\Hs\\
	=~& \frac{1}{2\dt}\left(\|\sigma_n\|_S^2 - \|\sigmas{n}\|_S^2
	+ \|\sigma_n - \sigmas{n}\|_S^2
	+ \frac{1}{a}\left(\|\alpha_n\|_\Hs^2 - \|\alpha_{n-1}\|_\Hs^2
	+ \|\alpha_n - \alpha_{n-1}\|_\Hs^2\right)\right).
\end{aligned}\end{align}
If we set $\sigmas{0}\defeq\sigma_0$ and $\alpha_{-1}\defeq\alpha_0$, then the right-hand side of \eqref{ineq:stab2} also vanishes for $n=0$.
Hence, adding \eqref{ineq:stab2} with $n$ replaced by $n-1$ to \eqref{ineq:stab1}, we eliminate $\|\sigma_{n-1}\|_S^2$:
\begin{align*}
  &\|v_n\|_\Hv^2 - \|v_{n-1}\|_\Hv^2
  + \|\sigmas{n}\|_S^2 - \|\sigmas{n-1}\|_S^2
  +\frac{1}{a}\left(\|\alpha_{n-1}\|_\Hs^2 - \|\alpha_{n-2}\|_\Hs^2\right)
  + \eta \dt\|\E(v_n)\|_\Hs^2 + \|v_n - v_{n-1}\|_\Hv^2\\
  \le~& \dt\left(\frac{c_K^2}{\eta}\|f_n\|_{\Vv^*}^2 + 2c_S\|h_n\|_\Hs^2 + \frac{1}{2}\|\sigmas{n}\|_S^2\right),
\end{align*}
By summing up for $n = 1,2,\ldots,m$, where $m\le N$, we obtain
\begin{align*}\begin{aligned}
    &\|v_m\|_\Hv^2 - \|v_0\|_\Hv^2
    + \|\sigmas{m}\|_S^2 - \|\sigma_0\|_S^2
    +\frac{1}{a}\left(\|\alpha_{m-1}\|_\Hs^2 - \|\alpha_0\|_\Hs^2\right)
    + \eta\dt\sum_{n=1}^m\|\E(v_n)\|_\Hs^2 
    + \sum_{n=1}^m\|v_n - v_{n-1}\|_\Hv^2\\
    \le~& \dt\sum_{n=1}^m\left(\frac{c_K^2}{\eta}\|f_n\|_{\Vv^*}^2 + 2c_S\|h_n\|_\Hs^2 + \frac{1}{2}\|\sigmas{n}\|_S^2\right),
\end{aligned}\end{align*}
where we have used $\sigmas{0} = \sigma_0$ and $\alpha_{-1} = \alpha_0$.
By the discrete Gronwall inequality and Lemma \ref{Lem:Korn}, if $\dt\le1$, then it holds that for all $m=1,2,\ldots,N$,
\begin{align*}\begin{aligned}
    &\|v_m\|_\Hv^2 + \|\sigmas{m}\|_S^2 
    +\frac{1}{a}\|\alpha_{m-1}\|_\Hs^2
    + \frac{\eta}{c_K^2}\dt\sum_{n=1}^m\|v_n\|_\Vv ^2 
    + \sum_{n=1}^m\|v_n - v_{n-1}\|_\Hv^2\\
    \le~& e\Bigg(\|v_0\|_\Hv^2 + \|\sigma_0\|_S^2 
    + \frac{1}{a}\|\alpha_0\|_\Hs^2
    + \dt\sum_{n=1}^m\left(\frac{c_K^2}{\eta}\|f_n\|_{\Vv^*}^2 + 2c_S\|h_n\|_\Hs^2\right)\Bigg).
\end{aligned}\end{align*}
Since we have that for all $n=1,2,\ldots,N$,
\begin{align*}\begin{aligned}
    \|f_n\|_{\Vv^*}^2
    &= \left\|\frac{1}{\dt}\int_{t_{n-1}}^{t_n}f(s)ds\right\|_{\Vv^*}^2
    \le \frac{1}{\dt}\|f\|_{L^2(t_{n-1},t_n;\Vv^*)}^2,\quad
    \|h_n\|_{\Hs}^2
    \le \frac{1}{\dt}\|h\|_{L^2(t_{n-1},t_n;\Hs)}^2,
\end{aligned}\end{align*}
dropping the nonnegative term $\frac{1}{a}\|\alpha_{m-1}\|_\Hs^2$, we obtain that for all $t\in[0,T]$,
\begin{align}\label{ineq:stabfinal1}\begin{aligned}
    &\|\bar{v}_\dt(t)\|_\Hv^2 
    + \|\bar{\sigma}_\dt^*(t)\|_S^2
    + \int_0^t\left(\|\bar{v}_\dt(s)\|_\Vv^2
    + \frac{1}{\dt}\|\hat{v}_\dt(s)-\bar{v}_\dt(s)\|_\Hv^2 \right)ds\\
    \le~& c_1(\|v_0\|_\Hv^2 + \|\sigma_0\|_S^2 + \|\alpha_0\|_\Hs^2 + \|f\|_{L^2(0,T;\Vv^*)}^2 + \|h\|_{L^2(0,T;\Hs)}^2)
\end{aligned}\end{align}
where $c_1=e\max\{1,\frac{1}{a},\frac{c_K^2}{\eta},2c_S\}\times\max\{1,a,\frac{c_K^2}{\eta}\}$. 

On the other hand, by adding \eqref{ineq:stab2} to \eqref{ineq:stab1}, we eliminate the term $\|\sigmas{n}\|_S^2$ on the left-hand side. 
Dropping the nonnegative terms $\eta \dt\|\E(v_n)\|_\Hs^2$ and $\|v_n-v_{n-1}\|_\Hv^2$, we obtain
\begin{align*}
    &\|v_n\|_\Hv^2 - \|v_{n-1}\|_\Hv^2
    + \|\sigma_{n}\|_S^2 - \|\sigma_{n-1}\|_S^2
    +\frac{1}{a}\left(\|\alpha_n\|_\Hs^2 - \|\alpha_{n-1}\|_\Hs^2\right)\\
    &+ \frac{1}{3}\|\sigma_n - \sigma_{n-1}\|_S^2
    + \frac{1}{3}\|\sigma_n - \sigmas{n}\|_S^2
    + \frac{1}{a}\|\alpha_n - \alpha_{n-1}\|_\Hs^2\\
    \le~& \dt\left(\frac{c_K^2}{\eta}\|f_n\|_{\Vv^*}^2 + 2c_S\|h_n\|_\Hs^2 + \frac{1}{2}\|\sigmas{n}\|_S^2\right),
\end{align*}
where we have used $\|\sigma_n-\sigmas{n}\|_S^2 + \|\sigma_n-\sigma_{n-1}\|_S^2 \le 3\|\sigma_n-\sigmas{n}\|_S^2 + 2\|\sigmas{n}-\sigma_{n-1}\|_S^2 \le 3(\|\sigma_n-\sigmas{n}\|_S^2 + \|\sigmas{n}-\sigma_{n-1}\|_S^2)$.
By summing up for $n = 1,2,\ldots,m$, where $m\le N$, we obtain
\begin{align*}\begin{aligned}
    &\|v_m\|_\Hv^2 - \|v_0\|_\Hv^2
    + \|\sigma_{m}\|_S^2 - \|\sigma_{0}\|_S^2
    +\frac{1}{a}\left(\|\alpha_m\|_\Hs^2 - \|\alpha_{0}\|_\Hs^2\right)\\
    &+ \sum_{n=1}^m\left(\frac{1}{3}\|\sigma_n - \sigma_{n-1}\|_S^2
    + \frac{1}{3}\|\sigma_n - \sigmas{n}\|_S^2
    + \frac{1}{a}\|\alpha_n - \alpha_{n-1}\|_\Hs^2\right)\\
    \le~& \dt\sum_{n=1}^m\left(\frac{c_K^2}{\eta}\|f_n\|_{\Vv^*}^2 + 2c_S\|h_n\|_\Hs^2 + \frac{1}{2}\|\sigmas{n}\|_S^2\right).
\end{aligned}\end{align*}
Hence, it holds that for all $t\in[0,T]$,
\begin{align}\label{ineq:stabfinal2}\begin{aligned}
    &\|\bar{v}_\dt(t)\|_\Hv^2 
    + \|\bar{\sigma}_\dt(t)\|_S^2
    + \|\bar{\alpha}_\dt(t)\|_\Hs^2\\
    &+ \frac{1}{\dt}\int_0^t\left(\|\hat{\sigma}_\dt(s)-\bar{\sigma}_\dt(s)\|_S^2
    + \|\bar{\sigma}_\dt(s)-\bar{\sigma}^*_\dt(s)\|_S^2
    + \|\hat{\alpha}_\dt(s)-\bar{\alpha}_\dt(s)\|_\Hs^2 \right)ds\\
    \le~& c_2(\|v_0\|_\Hv^2 + \|\sigma_0\|_S^2 + \|\alpha_0\|_\Hs^2 + \|f\|_{L^2(0,T;\Vv^*)}^2 + \|h\|_{L^2(0,T;\Hs)}^2+ \|\bar{\sigma}^*_\dt\|^2_{L^2(0,T;\Hs)}),
\end{aligned}\end{align}
where $c_2=\max\{\frac{c_K^2}{\eta},2c_S,\frac{1}{2}\}\times\max\{a,3\}$. 

Furthermore, by the first equation of (P) and the Korn inequality, it holds that 
\begin{align}\label{ineq:stabfinal3}\begin{aligned}
    \left\|\hat{v}'_\dt\right\|_{L^2(0,T;\Vv^*)}^2
    =~& \dt\sum_{n=1}^N \left(\sup_{0 \ne \varphi \in \Vv}\frac{1}{\|\varphi\|_\Vv}\left|\inner{\frac{v_n-v_{n-1}}{\dt}}{\varphi}{\Hv}\right|\right)^2 \\
    =~& \dt\sum_{n=1}^N \left(\sup_{0 \ne \varphi \in \Vv}\frac{\left|
        - \eta (\E(v_n),\E(\varphi))_\Hs - (\sigmas{n},\E(\varphi))_\Hs 
        + \blaket{f_n}{\varphi}\right|}{\|\varphi\|_\Vv}\right)^2\\
    \le~& \dt\sum_{n=1}^N \left(\sup_{0 \ne \varphi \in \Vv}\frac{
        \eta\|\E(v_n)\|_\Hs\|\E(\varphi)\|_\Hs 
    + \|\sigmas{n}\|_\Hs\|\E(\varphi)\|_\Hs 
    + \|f_n\|_{\Vv^*}\|\varphi\|_\Vv}{\|\varphi\|_\Vv}\right)^2\\
    \le~& \dt\sum_{n=1}^N \left(c_K\eta\|\E(v_n)\|_\Hs + c_K\|\sigmas{n}\|_\Hs + \|f_n\|_{\Vv^*}\right)^2\\
    \le~& 3c_K^2\eta^2\|\E(\bar{v}_\dt)\|_{L^2(0,T;\Hs)}^2 + 3c_K^2\|\bar{\sigma}_\dt^*\|_{L^2(0,T;\Hs)}^2 + 3\|f\|_{L^2(0,T;\Vv^*)}^2.
\end{aligned}\end{align}
By the equivalence of the norms $\|\cdot\|_S$ and $\|\cdot\|_\Hs$ provided by Lemma \ref{Lem:posSC}, (i) follows from \cref{ineq:stabfinal1,ineq:stabfinal2,ineq:stabfinal3}.

Next, we show (ii). Putting $\tau \defeq \Pop{g_n}(\sigma_{n-1} - \alpha_{n-1})$ in Lemma \ref{Lem:Sineq}, we have
\begin{align*}
    0 \ge~& \frac{1}{\dt}(\xi_n, \Pop{g_n}(\sigma_{n-1} - \alpha_{n-1}) - (\sigma_n-\alpha_n))_\Hs\\
    =~& \frac{1}{\dt}(\xi_n, \Pop{g_n}(\sigma_{n-1} - \alpha_{n-1}) 
    - \Pop{g_{n-1}}(\sigma_{n-1} - \alpha_{n-1}) 
    + (\sigma_{n-1} - \alpha_{n-1}) - (\sigma_n-\alpha_n))_\Hs\\
    =~& \left(\xi_n, \frac{\Pop{g_n}(\sigma_{n-1} - \alpha_{n-1}) 
    - \Pop{g_{n-1}}(\sigma_{n-1} - \alpha_{n-1})}{\dt} 
    - \frac{\sigma_n - \sigma_{n-1}}{\dt}
    + \frac{\alpha_n - \alpha_{n-1}}{\dt} \right)_\Hs,
\end{align*}
where we have used $\Pop{g_{n-1}}(\sigma_{n-1} - \alpha_{n-1}) = \sigma_{n-1} - \alpha_{n-1}$.
For the first term, by Proposition \ref{Prop:Phi}, we have 
\begin{align*}
    &\left|\left(\xi_n, 
    \frac{\Pop{g_n}(\sigma_{n-1} - \alpha_{n-1}) - \Pop{g_{n-1}}(\sigma_{n-1} - \alpha_{n-1})}{\dt} \right)_\Hs\right|
    \le \|\xi_n\|_\Hs \left\|\frac{g_n - g_{n-1}}{\dt}\right\|_{L^2(\Omega)}\\
    \le~& \frac{1}{4c_C}\|\xi_n\|_\Hs^2 
    + c_C\left\|\frac{g_n - g_{n-1}}{\dt}\right\|_{L^2(\Omega)}^2.
\end{align*}
On the other hand, from the 2nd equation of (P), it holds that 
\begin{align*}
    \|\E(v_n) + h_n\|_C^2
    &= \into \left(C\left(\xi_n + S\frac{\sigma_n - \sigma_{n-1}}{\dt}\right)\right):\left(\xi_n + S\frac{\sigma_n - \sigma_{n-1}}{\dt}\right) dx\\
    &= \|\xi_n\|_C^2 
    + 2\left(\xi_n, \frac{\sigma_n - \sigma_{n-1}}{\dt}\right)_\Hs 
    + \left\|\frac{\sigma_n - \sigma_{n-1}}{\dt}\right\|_S^2.
\end{align*}
Hence, for the second term and Lemma \ref{Lem:posSC}, we obtain
\begin{align*}
    - \left(\xi_n, \frac{\sigma_n - \sigma_{n-1}}{\dt}\right)_\Hs 
    =~& \frac{1}{2}\left( \|\xi_n\|_C^2 
    + \left\|\frac{\sigma_n - \sigma_{n-1}}{\dt}\right\|_S^2
    - \|\E(v_n) + h_n\|_C^2\right)\\
    \ge~& \frac{1}{2c_C}\|\xi_n\|_\Hs^2 
    + \frac{1}{2c_S}\left\|\frac{\sigma_n - \sigma_{n-1}}{\dt}\right\|_\Hs^2
    - c_C c_K^2\|v_n\|_\Vv^2 - c_C\|h_n\|_\Hs^2.
\end{align*}
For the third term, the third equation of (P) implies
\begin{align*}
    \left(\xi_n, \frac{\alpha_n - \alpha_{n-1}}{\dt}\right)_\Hs
    = \frac{1}{a}\left\|\frac{\alpha_n - \alpha_{n-1}}{\dt}\right\|_\Hs^2.
\end{align*}
Summarizing
\begin{align*}
    &\frac{1}{4c_C}\|\xi_n\|_\Hs^2 
    + \frac{1}{2c_S}\left\|\frac{\sigma_n - \sigma_{n-1}}{\dt}\right\|_\Hs^2
    + \frac{1}{a}\left\|\frac{\alpha_n - \alpha_{n-1}}{\dt}\right\|_\Hs^2
    \le c_C\left(\left\|\frac{g_n - g_{n-1}}{\dt}\right\|_{L^2(\Omega)}^2
    + c_K^2\|v_n\|_\Vv^2 + \|h_n\|_\Hs^2\right).
\end{align*}
By summing up for $n = 1,2,\ldots,N$, we obtain for all $t\in[0,T]$,
\begin{align*}\begin{aligned}
    &\|\bar{\xi}_\dt\|_{L^2(0,T;\Hs)}^2 
    + \left\|\hat{\sigma}'_\dt\right\|_{L^2(0,T;\Hs)}^2
    + \left\|\hat{\alpha}'_\dt\right\|_{L^2(0,T;\Hs)}^2\\
    \le~& c_3\left(\|g\|_{H^1(0,T;L^2(\Omega))}^2
    + \|\bar{v}_\dt\|_{L^2(0,T;\Vv)}^2 
    + \|\bar{h}_\dt\|_{L^2(0,T;\Hs)}^2\right),
\end{aligned}\end{align*}
where $c_3=c_C\max\{c_K^2,1\}\times\max\{4c_C,2c_S,a\}$ and we have used
\begin{align*}
    \left\|\frac{g_n - g_{n-1}}{\dt}\right\|_{L^2(\Omega)}^2
    \le \left\|\frac{1}{\dt}\int_{t_{n-1}}^{t_n}g'(s)ds\right\|_{L^2(\Omega)}^2
    \le \frac{1}{\dt}\|g\|_{H^1(t_{n-1},t_n;L^2(\Omega))}^2
\end{align*}
for all $n=1,2,\ldots,N$. Therefore, we obtain (ii).

\subsection{Existence and uniqueness of the solution to Problem \ref{Prob}}

In this subsection, we prove Theorem \ref{Th:exist}.

(Existence) By Theorem \ref{Th:stab}, there exist a sequence $(\dt_k)_{k\in\N}$ and three functions 
$v \in H^1(0,T;\Vv^*)\cap L^2(0,T;\Vv)$ (in particular, $v\in C([0,T];\Hv)$), $\sigma \in H^1(0,T;\Hs)$, and $\alpha \in H^1(0,T;\Hs)$ such that $\dt_k\rightarrow 0$ and for all $t\in[0,T]$
\begin{align}\label{convvhwh}
    \hat{v}_{\dt_k}\rightharpoonup v
    &\quad\text{weakly in }H^1(0,T;\Vv^*),\\
    &\quad\begin{aligned}\label{convvhwl}
        \text{weakly star in }L^\infty([0,T];\Hv),
    \end{aligned}\\
    \hat{v}_{\dt_k}(t) \rightarrow v(t)
    &\quad\begin{aligned}\label{convvhwc}
        \text{strongly in }\Hv,
    \end{aligned}\\
    \bar{v}_{\dt_k}\rightharpoonup v
    &\quad\begin{aligned}\label{convvbwl}
        \text{weakly in }L^2(0,T;\Vv),
    \end{aligned}\\
    \hat{\sigma}_{\dt_k} \rightharpoonup \sigma
    &\quad\begin{aligned}\label{convshwh}
        \text{weakly in }H^1(0,T;\Hs),
    \end{aligned}\\
    \hat{\sigma}_{\dt_k}(t) \rightharpoonup \sigma(t)
    &\quad\begin{aligned}\label{convshwc}
        \text{weakly in }\Hs,
    \end{aligned}\\
    \bar{\sigma}^*_{\dt_k} \rightharpoonup \sigma
    &\quad\begin{aligned}\label{convsswl}
        \text{weakly star in }L^\infty(0,T;\Hs),
    \end{aligned}\\
    \bar{\sigma}_{\dt_k} \rightharpoonup \sigma
    &\quad\begin{aligned}\label{convsbwl}
        \text{weakly star in }L^\infty(0,T;\Hs),
    \end{aligned}\\
    \hat{\alpha}_{\dt_k} \rightharpoonup \alpha
    &\quad\begin{aligned}\label{convahwh}
        \text{weakly in }H^1(0,T;\Hs),
    \end{aligned}\\
    \hat{\alpha}_{\dt_k}(t) \rightharpoonup \alpha(t)
    &\quad\begin{aligned}\label{convahwc}
        \text{weakly in }\Hs,
    \end{aligned}\\
    \bar{\alpha}_{\dt_k} \rightharpoonup \alpha
    &\quad\begin{aligned}\label{convabwl}
        \text{weakly star in }L^\infty(0,T;\Hs),
    \end{aligned}\\
    \bar{\xi}_{\dt_k}\rightharpoonup \xi
    &\quad\begin{aligned}\label{convxbwl}
        \text{weakly in }L^2(0,T;\Hs),
    \end{aligned}
\end{align}
as $k\rightarrow\infty$. It should be noted that 
    $(\hat{v}_{\dt_k})_{k\in\N}$ and $(\bar{v}_{\dt_k})_{k\in\N}$ possess a common limit function $v$ and 
    $(\hat{\sigma}_{\dt_k})_{k\in\N}$, $(\bar{\sigma}^*_{\dt_k})_{k\in\N}$ and $(\bar{\sigma}_{\dt_k})_{k\in\N}$ possess a common limit function $\sigma$ and 
    $(\hat{\alpha}_{\dt_k})_{k\in\N}$ and $(\bar{\alpha}_{\dt_k})_{k\in\N}$ possess a common limit function $\alpha$. 
Indeed, the weak convergences \cref{convvhwh,convvhwl,convvbwl,convshwh,convshwc,convsswl,convsbwl,convahwh,convahwc,convabwl,convxbwl} immediately follow by Theorem \ref{Th:stab} and Lemma \ref{Lem:AA}.
If we define $\hat{v}_\dt^\circ \in C([0,T];\Vv)$ by 
\[
    \hat{v}_\dt^\circ(t)\defeq\left\{\begin{aligned}
        &v_1 &&\mbox{ if }t\in[0,\dt],\\
        &\hat{v}_\dt(t) &&\mbox{ if }t\in[\dt,T],
    \end{aligned}\right.
\]
then we obtain that $(\hat{v}_\dt^\circ)_{0<\dt<1}$ is bounded in $H^1(0,T;\Vv^*)$ and $L^\infty(0,T;\Hv)$, and hence, 
\[
    \hat{v}_{\dt_k}^\circ \rightarrow v 
    \quad\mbox{strongly in }C([0,T];\Hv)
\]
as $k\rightarrow\infty$, by the Aubin--Lions theorem \cite[Theorem II.5.16]{BF13}.
For all $t\in(0,T]$, there exists $l\in\N$ such that $\dt_l \le t$. Since it holds that $\hat{v}_{\dt_k}^\circ(t)=\hat{v}_{\dt_k}(t)$ for all $k \ge l$, \eqref{convvhwc} holds.

Next, we demonstrate that the limit functions $(v, \sigma, \alpha)$ satisfy \eqref{sys:Prob} and $\sigma(t) \in K(t;\alpha)$ for all $t \in [0, T]$.
By the fourth equation of (P) with $\dt\defeq\dt_k$, it holds that for all $t\in[t_{n-1},t_n],n=1,2,\ldots,N$,
\[\begin{aligned}
    |(\hat{\sigma}_{\dt_k}(t) - \hat{\alpha}_{\dt_k}(t))^D| 
    &= \left|\frac{t-t_{n-1}}{\dt}(\sigma_n-\alpha_n)^D 
    + \frac{t_n-t}{\dt}(\sigma_{n-1} - \alpha_{n-1})^D\right|\\
    &\le \frac{t-t_{n-1}}{\dt}g_n + \frac{t_n-t}{\dt}g_{n-1}
    = \hat{g}_{\dt_k}(t)
\end{aligned}\]
a.e. on $\Omega$, i.e. $\hat{\sigma}_{\dt_k}(t) - \hat{\alpha}_{\dt_k}(t) \in \left\{\tau\in\Hs : |\tau^D|\le \hat{g}_{\dt_k}(t) \mbox{ a.e. in }\Omega\right\}$ for all $t\in[0,T]$. 
By $g\in H^1(0,T;\Lo)$, we have that $\hat{g}_\dt \rightarrow g$ strongly in $C([0,T];\Lo)$. By the Riesz--Fischer theorem, for each \(t\in[0,T]\) we have \(\hat g_{\dt_k}(t)\to g(t)\) a.e. in \(\Omega\) along a (possibly \(t\)-dependent) subsequence\footnote{
    One cannot, in general, extract a single subsequence yielding a.e. convergence for every \(t\in[0,T]\).
    One may first obtain a.e. convergence on a countable dense set of \([0,T]\) by a diagonal argument, and then extend the constraint $\sigma(t)\in K(t;\alpha)$ to all \(t\in[0,T]\) by using the continuity \(g\in C([0,T];L^2(\Omega))\) and \(\sigma-\alpha\in C([0,T];\Hs)\).
}. For fixed $k\in\N$ and $t\in[0,T]$, since 
\[
    \left\{\tau\in\Hs : |\tau^D|\le \sup_{l \ge k}\hat{g}_{\dt_l}(t) \mbox{ a.e. in }\Omega\right\}    
\] 
is closed and convex in the strong topology of $\Hs$, it is also closed and convex in the weak topology of $\Hs$. By \eqref{convshwc}, we obtain for all $k\in\N$,
\[
    \sigma(t)-\alpha(t) \in \left\{\tau\in\Hs : |\tau^D|\le \sup_{l \ge k}\hat{g}_{\dt_l}(t) \mbox{ a.e. in }\Omega\right\},
\]
and hence, 
\[\begin{aligned}
    \sigma(t)-\alpha(t) &\in \bigcap_{k\in\N}\left\{\tau\in\Hs : |\tau^D|\le \sup_{l \ge k}\hat{g}_{\dt_l}(t) \mbox{ a.e. in }\Omega\right\}\\
    &= \left\{\tau\in\Hs : |\tau^D|\le \inf_{k\in\N}\sup_{l \ge k}\hat{g}_{\dt_l}(t) \mbox{ a.e. in }\Omega\right\}\\
    &= \left\{\tau\in\Hs : |\tau^D|\le g(t) \mbox{ a.e. in }\Omega\right\}
    = \tilde{K}(t),
\end{aligned}\]
i.e. $\sigma(t)\in K(t;\alpha)$ for all $t\in[0,T]$.

By the first equation of (P) with $\dt\defeq\dt_k$, it holds that for all $\varphi\in\Vv$ and $\theta\in C^\infty_0(0,T)$,
\begin{align*}
    &\int_0^T\left(\blaket{\hat{v}'_{\dt_k}}{\theta\varphi}
    + \eta(\E(\bar{v}_{\dt_k}), \E(\theta\varphi))_\Hs + (\bar{\sigma}^*_{\dt_k}, \E(\theta\varphi))_\Hs \right)dt
    = \int_0^T\blaket{\bar{f}_{\dt_k}}{\theta\varphi} dt.
\end{align*}
By taking $k\rightarrow\infty$, we obtain that for all $\varphi\in\Vv$ and $\theta\in C^\infty_0(0,T)$,
\begin{align*}
    \int_0^T\left(\blaket{v'}{\theta\varphi}
    + \eta(\E(v), \E(\theta\varphi))_\Hs + (\sigma, \E(\theta\varphi))_\Hs \right)dt
    = \int_0^T\blaket{f}{\theta\varphi} dt,
\end{align*}
which implies the first equation of \eqref{sys:Prob}. 

The third equation of (P) with $\dt\defeq\dt_k$ directly implies the third equation of \eqref{sys:Prob} by taking $k\rightarrow\infty$. Similarly, the definition of $\xi_n$ in (P) with $\dt\defeq\dt_k$ directly implies $\xi = \E(v) + h - S\sigma'$ by taking $k\rightarrow\infty$.

By Lemma \ref{Lem:Sineq} and the third equation of (P), it holds that for all $t \in (t_{n-1},t_n), n=1,2,\ldots,N$ and $\tau \in C([0,T];\Hs)$ with $\tau(t) \in K(t;\alpha)$,
\begin{align*}
    0 \ge~& (-\xi_n, (\sigma_n - \alpha_n) - (\tau(t_n) - \alpha(t_n)))_\Hs\\
    =~& \left(S\frac{\sigma_n - \sigma_{n-1}}{\dt} - \E(v_n) - h_n, 
    \sigma_n - \tau(t_n)\right)_\Hs
    + \frac{1}{a}\left(\frac{\alpha_n-\alpha_{n-1}}{\dt}, \alpha_n\right)_\Hs
    - (\xi_n,\alpha(t_n))_\Hs\\
    =~& \left(S\hat{\sigma}'_{\dt_k}(t) 
    - \E(\bar{v}_{\dt_k}(t)) - \bar{h}_{\dt_k}(t), 
    \bar{\sigma}_{\dt_k}(t) - \tau_{\dt_k}(t)\right)_\Hs\\
    &+ \frac{1}{2a\dt}(\|\hat{\alpha}_{\dt_k}(t_n)\|_\Hs^2 
    - \|\hat{\alpha}_{\dt_k}(t_{n-1})\|_\Hs^2)
    - (\bar{\xi}_{\dt_k}(t),\alpha_{\dt_k}(t))_\Hs\\
    \ge~& \left(S\hat{\sigma}'_{\dt_k}(t) 
    - \E(\bar{v}_{\dt_k}(t)) - \bar{h}_{\dt_k}(t), 
    \bar{\sigma}_{\dt_k}(t) - \tau_{\dt_k}(t)\right)_\Hs\\
    &+ \frac{1}{2a(t-t_{n-1})}(\|\hat{\alpha}_{\dt_k}(t)\|_\Hs^2 
    - \|\hat{\alpha}_{\dt_k}(t_{n-1})\|_\Hs^2)
    - (\bar{\xi}_{\dt_k}(t),\alpha_{\dt_k}(t))_\Hs,
\end{align*}
where $\tau_{\dt_k}(t) \defeq \tau(t_n)$ and $\alpha_{\dt_k}(t) \defeq \alpha_n$ for $t\in(t_{n-1},t_n]$, $n=1,2,\ldots,N$.
By multiplying by $t-t_{n-1}$, we obtain that for all $t\in[t_{n-1},t_n], n=1,2,\ldots,N$,
\begin{align*}
    &\int_{t_{n-1}}^{t} \left(S\hat{\sigma}'_{\dt_k} - \E(\bar{v}_{\dt_k}) - \bar{h}_{\dt_k}, 
    \bar{\sigma}_{\dt_k} - \tau_{\dt_k}\right)_\Hs ds\\
    +~& \frac{1}{2a}(\|\hat{\alpha}_{\dt_k}(t)\|_\Hs^2 
    - \|\hat{\alpha}_{\dt_k}(t_{n-1})\|_\Hs^2)
    - \int_{t_{n-1}}^{t} (\bar{\xi}_{\dt_k}, \alpha_{\dt_k})_\Hs ds
    \le 0,
\end{align*}
and hence, it holds that for all $t\in[0,T]$,
\begin{align*}
    &\int_0^t \left(S\hat{\sigma}'_{\dt_k} - \E(\bar{v}_{\dt_k}) - \bar{h}_{\dt_k}, 
    \bar{\sigma}_{\dt_k} - \tau_{\dt_k}\right)_\Hs ds\\
    +~& \frac{1}{2a}(\|\hat{\alpha}_{\dt_k}(t)\|_\Hs^2 - \|\alpha_0\|_\Hs^2)
    - \int_0^t (\bar{\xi}_{\dt_k}, \alpha_{\dt_k})_\Hs ds
    \le 0.
\end{align*}
Since $\tau_\dt$, $\bar{h}_\dt$, and $\alpha_\dt$ converge to $\tau$, $h$, and $\alpha$, respectively, strongly in $L^2(0,T;\Hs)$, we have that for each fixed $t\in[0,T]$,
\[\begin{aligned}
    \int_0^t \left(S\hat{\sigma}'_{\dt_k} - \E(\bar{v}_{\dt_k}) - \bar{h}_{\dt_k}, \tau_{\dt_k}\right)_\Hs ds
    &\rightarrow \int_0^t \left(S\sigma' - \E(v) - h,\tau\right)_\Hs ds,\\
    \int_0^t \left(\bar{h}_{\dt_k}, 
    \bar{\sigma}_{\dt_k}\right)_\Hs ds
    &\rightarrow \int_0^t (h, \sigma)_\Hs ds,\\
    \int_0^t (\bar{\xi}_{\dt_k}, \alpha_{\dt_k})_\Hs ds
    &\rightarrow \int_0^t (\xi, \alpha)_\Hs ds,
\end{aligned}\] 
as $k \rightarrow \infty$. 
Furthermore, by Theorem \ref{Th:stab}, \cref{convshwc,convvhwc,convahwh}, we obtain that
\[\begin{aligned}
    &\liminf_{k\rightarrow\infty}
    \int_0^t \left(S\hat{\sigma}'_{\dt_k}, 
    \bar{\sigma}_{\dt_k}\right)_\Hs ds\\
    =~& \frac{1}{2}\liminf_{k\rightarrow\infty}
    \int_0^t \frac{d}{dt}\|\hat{\sigma}_{\dt_k}\|_S^2ds
    +\liminf_{k\rightarrow\infty}
    \int_0^t \left(S\hat{\sigma}'_{\dt_k}, 
    \bar{\sigma}_{\dt_k}-\hat{\sigma}_{\dt_k}\right)_\Hs ds\\
    =~& \frac{1}{2}\liminf_{k\rightarrow\infty}
    (\|\hat{\sigma}_{\dt_k}(t)\|_S^2 - \|\sigma_0\|_S^2)
    + \liminf_{k\rightarrow\infty}
    \int_0^t \left(S\hat{\sigma}'_{\dt_k}, 
    \bar{\sigma}_{\dt_k}-\hat{\sigma}_{\dt_k}\right)_\Hs ds\\
    \ge~& \frac{1}{2}\|\sigma(t)\|_S^2
    - \frac{1}{2}\|\sigma_0\|_S^2 + 0
    = \int_0^t \left(S\sigma', \sigma\right)_\Hs ds,
\end{aligned}\]
\[\begin{aligned}
    &\liminf_{k\rightarrow\infty}
    \left(-\int_0^t \left(\E(\bar{v}_{\dt_k}), 
    \bar{\sigma}_{\dt_k}\right)_\Hs ds\right)\\
    =~& \liminf_{k\rightarrow\infty}
    \left(-\int_0^t \left(\E(\bar{v}_{\dt_k}), 
    \bar{\sigma}^*_{\dt_k}\right)_\Hs ds
    + \int_0^t \left(\E(\bar{v}_{\dt_k}), 
    \bar{\sigma}^*_{\dt_k} - \bar{\sigma}_{\dt_k}\right)_\Hs ds\right)
    \\
    =~& \liminf_{k\rightarrow\infty}
    \int_0^t\left(\blaket{\hat{v}'_{\dt_k}}{\bar{v}_{\dt_k}}
    + \eta\|\E(\bar{v}_{\dt_k})\|_\Hs^2
    - \blaket{\bar{f}_{\dt_k}}{\bar{v}_{\dt_k}} \right)ds\\
    %
    %
    \ge~& \frac{1}{2}\|v(t)\|_\Hv^2 
    - \frac{1}{2}\|v(0)\|_\Hv^2
    + \int_0^t\left(\eta\|\E(v)\|_\Hs^2
    - \blaket{f}{v} \right)ds
    = -\int_0^t \left(\E(v), \sigma\right)_\Hs ds,\\
    & \liminf_{k\rightarrow\infty}\frac{1}{2a}(\|\hat{\alpha}_{\dt_k}(t)\|_\Hs^2 - \|\alpha_0\|_\Hs^2)\\
    \ge~& \frac{1}{2a}(\|\alpha(t)\|_\Hs^2 - \|\alpha_0\|_\Hs^2)
    = \frac{1}{a}\int_0^t \left(\alpha', \alpha\right)_\Hs ds
    = \int_0^t (\xi, \alpha)_\Hs ds,
\end{aligned}\]
where we have used the first and third equations of \eqref{sys:Prob}. 
Therefore, we obtain that for all $t\in[0,T]$ and $\tau \in C([0,T];\Hs)$ with $\tau(t) \in K(t;\alpha)$,
\begin{align*}
    \int_0^t\left(S\sigma' - \E(v) - h,\sigma - \tau\right)_\Hs ds
    \le~& \liminf_{k\rightarrow\infty} 
    \int_0^t \left(S\hat{\sigma}'_{\dt_k} - \E(\bar{v}_{\dt_k}) - \bar{h}_{\dt_k}, 
    \bar{\sigma}_{\dt_k} - \tau_{\dt_k}\right)_\Hs ds\\
    &+ \frac{1}{2a}(\|\hat{\alpha}_{\dt_k}(t)\|_\Hs^2 - \|\alpha_0\|_\Hs^2)
    - \int_0^t (\bar{\xi}_{\dt_k}, \alpha_{\dt_k})_\Hs ds
    \le 0,
\end{align*}
which implies the second inequality of \eqref{sys:Prob}.

(Uniqueness) Let $(v_1, \sigma_1, \alpha_1)$ and $(v_2, \sigma_2, \alpha_2)$ be two solutions to Problem \ref{Prob}. If we set $e_v \defeq v_1-v_2$ and $e_\sigma \defeq \sigma_1-\sigma_2$, by the first equation of \eqref{sys:Prob}, we have for all $\varphi \in\Vv$,
\[
    \blaket{e'_v}{\varphi}
    + \eta (\E(e_v),\E(\varphi))_\Hs + (e_\sigma,\E(\varphi))_\Hs
    = 0.
\]
Putting $\varphi\defeq e_v$ and integrating over time, we obtain that for all $t\in[0,T]$,
\begin{align}\label{ineq:uniq1}
    \frac{1}{2}\|e_v(t)\|_\Hv^2
    + \int_0^t\left(\eta \|\E(e_v)\|_\Hs^2 + (e_\sigma,\E(e_v))_\Hs\right)ds
    = 0,
\end{align}
where we have used $e_v(0) = v_1(0) - v_2(0) = 0$.

Since $(v_1, \sigma_1, \alpha_1)$ is the solution to Problem \ref{Prob}, by putting $\tau \defeq \sigma_2 - \alpha_2 + \alpha_1$ in \eqref{sys:Prob}, we obtain that 
\begin{align}\label{ineq:uniq2}\begin{aligned}
    0 &\ge \left(S\sigma'_1 - \E(v_1) - h,\sigma_1 - \sigma_2 - \alpha_1 + \alpha_2\right)_\Hs
    = \left(S\sigma'_1 - \E(v_1) - h,\sigma_1 - \sigma_2\right)_\Hs
    + \left(\xi_1, \alpha_1 - \alpha_2\right)_\Hs\\
    &= \left(S\sigma'_1 - \E(v_1) - h,e_\sigma\right)_\Hs
    + \frac{1}{a}\left(\alpha'_1, \alpha_1 - \alpha_2\right)_\Hs,
\end{aligned}\end{align}
where $\xi_1 = \E(v_1) + h - S\sigma'_1$.
On the other hand, since $(v_2, \sigma_2, \alpha_2)$ is the solution to Problem \ref{Prob}, by putting $\tau \defeq \sigma_1 - \alpha_1 + \alpha_2$ in \eqref{sys:Prob}, we obtain 
\begin{align}\label{ineq:uniq3}\begin{aligned}
    0 &\ge \left(S\sigma'_2 - \E(v_2) - h,\sigma_2 - \sigma_1 + \alpha_1 - \alpha_2\right)_\Hs
    = \left(S\sigma'_2 - \E(v_2) - h,\sigma_2 - \sigma_1\right)_\Hs
    + \left(\xi_2, \alpha_2 - \alpha_1\right)_\Hs\\
    &= -\left(S\sigma'_2 - \E(v_2) - h,e_\sigma\right)_\Hs
    - \frac{1}{a}\left(\alpha'_2, \alpha_1 - \alpha_2\right)_\Hs,
\end{aligned}\end{align}
where $\xi_2 = \E(v_2) + h - S\sigma'_2$.
Adding \eqref{ineq:uniq2} and \eqref{ineq:uniq3} together, if we set $e_\alpha\defeq \alpha_1 - \alpha_2$ we get
\begin{align*}
    \left(Se'_\sigma-\E(e_v),e_\sigma\right)_\Hs
    + \frac{1}{a}\left(e'_\alpha, e_\alpha\right)_\Hs \le 0.
\end{align*}
Integrating in time, we obtain that for all $t\in[0,T]$,
\begin{align}\label{ineq:uniq4}
    \frac{1}{2}\|e_\sigma(t)\|_S^2 + \frac{1}{2a}\|e_\alpha(t)\|_\Hs^2
    - \int_0^t\left(\E(e_v),e_\sigma\right)_\Hs ds \le 0,
\end{align}
where we have used $e_\sigma(0) = \sigma_1(0) - \sigma_2(0) = 0$ and $e_\alpha(0) = \alpha_1(0) - \alpha_2(0) = 0$. Thus, summing up \eqref{ineq:uniq1} and \eqref{ineq:uniq4}, it holds that for all $t\in[0,T]$, 
\[
    \frac{1}{2}\|e_v(t)\|_\Hv^2 
    + \frac{1}{2}\|e_\sigma(t)\|_S^2
    + \frac{1}{2a}\|e_\alpha(t)\|_\Hs^2
    + \eta\int_0^t\|\E(e_v)\|_\Hs^2ds \le 0,
\]
and hence, $v_1-v_2 = e_v = 0$, $\sigma_1-\sigma_2 = e_\sigma = 0$, and $\alpha_1 - \alpha_2 = e_\alpha = 0$ on $[0,T]$.

\section{Conclusion}

We studied Problem \ref{Prob}, which describes a dynamical elasto-plasticity model with Kelvin--Voigt viscosity and linear kinematic hardening, where the yield surface depends on time and on the plastic strain $\varepsilon_p=\E(u)-S\sigma$ through the translated constraint set.
To construct solutions, we introduced the projection-based time discretization (P), in which one first solves a linear viscous--elastic subproblem to obtain a trial stress and then projects it onto the translated admissible set.
We proved that the discrete solutions produced by (P) satisfy stability estimates under suitable norms. On the basis of these uniform bounds, we passed to the limit as the time step tends to zero and obtained the existence of a weak solution to Problem \ref{Prob}; moreover, uniqueness follows from an energy argument.

Several directions remain for future work. First, although the present analysis is carried out in spatially continuous function spaces, the structure of (P) suggests a natural path to a fully discrete method: by replacing the function spaces with suitable finite element spaces, one can expect an implementable FEM-type algorithm, and it will be important to establish convergence of such fully discrete approximations and to perform numerical tests.
Second, it is of interest to investigate how far the projection-based strategy can be generalized, for instance to other yield criteria such as Tresca-type constraints, to more general convex admissible sets, or to extensions where the hardening parameter $a$ is replaced by a matrix-valued operator.
Finally, from the analytical viewpoint, it is natural to ask whether one can treat regimes beyond the present regularized setting. The case $\eta=0$ is particularly challenging since the viscous term no longer yields the spatial $H^1$-control of the velocity and one cannot handle $\E(v)$ within the same compactness framework. In addition, allowing $a<0$ corresponds to a softening-type effect \cite{Gudoshnikov25} and breaks the coercivity/monotonicity structure used in the stability and uniqueness arguments, so the well-posedness theory would require substantially different techniques.


\bibliographystyle{spmpsci}      
\bibliography{Proj4Plasticity.bib}

\end{document}